\documentclass[amstex,11pt]{article}
\usepackage{amsmath,amsfonts,amssymb,amsthm,enumerate,hyperref,multicol,color,graphicx}

\usepackage[labelformat=brace]{subcaption}
\DeclareCaptionFormat{myformat}{\hspace*{1cm}#1}
\usepackage{amsmath}
\newtheorem{lemma}{Lemma}
\newtheorem{thm}{Theorem}
\newtheorem{definition}{Definition}

\newtheorem{corollary}{Corollary}

\newtheorem{hypp}{Hypothesis}
\usepackage[autostyle]{csquotes}

\numberwithin{equation}{section}

\begin{document}

\begin{center}
\Large{\textbf{Instantaneous gelation and nonexistence for the  Oort-Hulst-Safronov coagulation model}}
\end{center}

\centerline{Pooja Rai${}^1$, Ankik Kumar Giri${}^2{^*}$ and Volker John${}^3$}\let\thefootnote\relax\footnotetext{${^*}$Corresponding author. Tel +91-1332-284818 (O);  Fax: +91-1332-273560  \newline{\it{${}$ \hspace{.3cm} Email address: }}ankik.giri@ma.iitr.ac.in}
\medskip
{\footnotesize

  \centerline{ ${}^{1,2}$Department of Mathematics, Indian Institute of Technology Roorkee,}
   \centerline{ Roorkee-247667, Uttarakhand, India}

\centerline{${}^{3}$Weierstrass Institute for Applied Analysis and Stochastics, Mohrenstr. 39, 10117 Berlin, Germany}
   \centerline{${}^{3}$Freie Universit\"{a}t of Berlin, Department of Mathematics and Computer Science, Arnimallee 6, 14195 Berlin, Germany}

}

\bigskip

%
%

\begin{abstract}
The possible occurrence of instantaneous gelation to Oort-Hulst-Safronov (OHS) coagulation equation is investigated for a certain class of unbounded coagulation kernels. The existence of instantaneous gelation is confirmed by showing the nonexistence of mass-conserving weak solutions. Finally, it is shown that for such kernels, there is no weak solution to the OHS coagulation equation at any time interval.
\end{abstract}



\section{Introduction}\label{existintroduction1}
Coagulation, a basic kinetic process, represents the dynamics of particle growth in which two or more particles adhere to form a new larger particle. This process may occur in several physical phenomenon such as fluidized bed granulation, planet formation, polymerization etc. A discrete system of differential equations to describe the coagulation of collides travelling in Brownian motion was first  developed by Smoluchowski \cite{Smoluchowski:1917} which is known as the Smoluchowski coagulation equation (SCE) and its continuous version was later given by M\"uller \cite{Muller:1928}. Each particle is recognized by its size (or volume) and the parameter value of size of each particle is either contained in the set of positive integer numbers $\mathbb{N}\setminus\{0\}$ (for the discrete case) or in the set of positive real numbers $\mathbb{R}_{>0}=(0, +\infty)$ (for the continuous case).  In a distinct sense, a different coagulation equation was proposed by Oort and Hulst \cite{Hulst:1946} that was adopted in astronomy to describe the coagulation of stellar objects when these objects merge irreversibly via binary interactions to configuration of a bigger objects at a specific moment. However, the tractable form of this equation was later given by Safronov \cite{Safronov:1972}. As a result, it is referred to as the Oort-Hulst-Safronov (OHS) coagulation equation. The discrete version of OHS equation is introduced by Dubvoksi\v{i} which was later termed as the Safronov-Dubovski\v{i} coagulation equation, see \cite{Dubovskii:1999, Dubovski:1999II, Bagland:2005}. The nonlinear nonlocal OHS coagulation equation for the evolution of the concentration $\xi(\mu,t)$ of particles of size $\mu \in \mathbb{R}_{>0}$ at time $t\geq0$ is given by

 \begin{align}\label{ConEq1}
\frac{\partial \xi(\mu,t)}{\partial t}  = & -\frac{\partial }{\partial \mu} \bigg( \xi(\mu, t) \int_0^{\mu} \nu \Lambda(\mu, \nu) \xi(\nu, t)\ d\nu \bigg)\nonumber\\
& -\int_{\mu}^{\infty} \Lambda(\mu, \nu) \xi(\mu, t) \xi(\nu, t)\ d\nu,\ \quad (\mu,t)\in \mathbb{R}_{>0}^2,\
\end{align}
with initial condition
\begin{equation}\label{ConEqin1}
\xi(\mu, 0) = \xi^{\mathrm{in}}(\mu)\ge 0,\ \quad \mu\in \mathbb{R}_{>0}.
\end{equation}


 The coagulation kernel $\Lambda(\mu,\nu)$ denotes the intensity force at which particles of size $\mu$ merge irreversibly with particles of size $\nu$ to form the larger particles that is assumed to be nonnegative and symmetric, (that is) i.e., $\Lambda(\mu,\nu)=\Lambda(\nu,\mu)\geq0$ for all $\mu,\nu \in \mathbb{R}_{>0}^2$. In \eqref{ConEq1}, $\frac{\partial}{\partial t}$ and $\frac{\partial}{\partial \mu}$ are the partial derivatives with respect to time and space, respectively.

  According to \cite{Dubovskii:1999}, in OHS equation \eqref{ConEq1}, the aggregation of particles of size $\mu$ with smaller particles alters the size of particles of size $\mu$. Furthermore, the coagulation of particle of size $\mu$ with bigger particles changes the number of particles of size $\mu$. Therefore, the first term on the right-hand side of \eqref{ConEq1} only changes $\xi(\mu,t)$ due to coagulation with small particles whereas the last term indicates the decay of particle of size $\mu$ due to coagulation with bigger particles.

  The total mass of the particles in the whole system, at any time $t\geq0$, is given by
$$\mathcal{M}^{1}(t)=\int_{0}^{\infty}\mu\xi(\mu,t)\ d\mu.$$
Formally, we know that the total mass of particles is neither originated nor ended by any reaction. Therefore, it is expected that the total mass remains conserved throughout time evolution, i.e.,
 \begin{align}\label{MassConseDisCon}
\int_{0}^{\infty}\mu\xi(\mu,t)\ d\mu=\int_{0}^{\infty}\mu\xi^{\mathrm{in}}(\mu)\ d\mu=\varrho_{0},\quad \forall~t\geq0.
 \end{align}
However, for the coagulation kernels growing sufficiently rapidly for large $\mu$, $\nu$ such as $\Lambda (\mu, \nu)= (\mu \nu)^{r/2}$ for $r \in (1,2]$, there is a possibility to have a runaway growth that can lead to the formation of particles of infinite mass in finite time.  These particles of infinite mass are called \emph{infinite gels} which are then removed from the system. As a result, we see that the mass conservation breaks down in finite time, i.e.,
 $$\int_{0}^{\infty}\mu\xi(\mu,t)\ d\mu<\int_{0}^{\infty}\mu\xi^{\mathrm{in}}(\mu)\ d\mu,\quad T_{\mathrm{gel}}<t,$$
 and this phenomenon is known as \emph{gelation}. Here, $T_{\mathrm{gel}}$ is starting time after which the mass conservation breaks down is called the \emph{gelation time} which can mathematically be defined as
\[
 T_{\mathrm{gel}}:= \inf\left\{t\geq 0 \ \ \mbox{such that} \ \ \mathcal{M}^{1}(t)<\mathcal{M}^{1}(0)=\varrho_{0}\right\}.
\]
 If the gelation time $T_{\mathrm{gel}} = 0$, then this phenomenon is known as \emph{instantaneous gelation}.

The gelation phenomenon for the SCE has been significantly discussed in the literature (see \cite{Banasiak} and reference therein). Furthermore, the occurrence of instantaneous gelation for the SCE was first investigated by Dongen \cite{Dongen:1987} but the first rigorous mathematical proof was introduced by Carr \&  da Costa \cite{Carr:1992}. Recently, Banasiak et al.\cite{Banasiak} have supplemented a proof which confirms that the instantaneous gelation takes place for the continuous version of SCE under certain classes of unbounded coagulation kernels. However, there are a very few articles available on the occurrence of gelation for the OHS coagulation equation, see for instance  \cite{Dubovskii:1999}, \cite{Lachowicz:2003}. More recently, Das and Saha \cite{Das:2022} studied the instantaneous gelation result to the discrete version of the OHS coagulation equation with a specific class of coagulation kernels. The idea of the their work is mainly motivated from \cite{Carr:1992}. To the best of our knowledge, the instantaneous gelation result to OHS coagulation equation \eqref{ConEq1}--\eqref{ConEqin1} has not been addressed in the literature till date. Therefore, the purpose of this paper is to investigate the occurrence of instantaneous gelation to \eqref{ConEq1}--\eqref{ConEqin1} under the class of coagulation kernels mentioned in hypothesis $\textbf{(A)}$ given below. In addition, we also show the nonexistence of weak solutions to \eqref{ConEq1}--\eqref{ConEqin1}  with the same class of kernels.

Before outlining the results of this paper, let us discuss the mathematical results available on the solutions to the OHS coagulation equation \eqref{ConEq1}--\eqref{ConEqin1}. Results on the existence, gelation, mass conservation and large time behavior of weak solutions to the OHS equation \eqref{ConEq1}--\eqref{ConEqin1} with unbounded coagulation kernels were first established by Lachowicz et al.\ \cite{Lachowicz:2003}. Moreover, they also established a deep connection between OHS coagulation equation and the continuous version of Smoluchowski coagulation equation by introducing a generalized coagulation equation. Next, in \cite{Bagland:2005}, Bagland  proved that a suitable sequence of solutions of discrete version of OHS equation converges towards a solution of OHS equations \eqref{ConEq1}--\eqref{ConEqin1}. In addition, he has also derived an explicit solution $\xi(t,\mu)=\frac{2}{M(1+t)^{2}}\textbf{1}_{[0,M]}\frac{\mu}{(1+t)}$ to the OHS equations \eqref{ConEq1}--\eqref{ConEqin1} with the coagulation kernel $\Lambda(\mu,\nu)=1$ and initial data $\xi(\mu,0)=\frac{2}{M}\textbf{1}_{[0,M]}$, where $M$ is a positive constant. Later, in \cite{Laurencot:2005} and \cite{Laurencot:2006}, Lauren\c{c}ot discussed the self similar solution to \eqref{ConEq1}--\eqref{ConEqin1} under the coagulation kernels $\Lambda(\mu,\nu)\equiv1$ and $\Lambda(\mu,\nu)=\mu\nu$, respectively. After these results, Bagland and Lauren\c{c}ot \cite{Bagland:2007} confirmed the presence of self similar solutions to \eqref{ConEq1}--\eqref{ConEqin1} under the coagulation kernel $\Upsilon(\mu,\nu)=\mu^{\alpha}+\nu^{\alpha},~\alpha\in(0,1)$, where the self-similar profiles for these self similar solutions are compactly supported. More recently, Barik et. al.\ \cite{Barik:2022} have proved the existence of mass-conserving solutions to \eqref{ConEq1}--\eqref{ConEqin1} under the assumptions that the coagulation kernels have an algebraic singularity near zero and that they grow linearly at infinity.

The plan of the paper is the following:  Section 2 presents notations of space, definition and hypothesis which are needed in subsequent sections. At the end of this section, the main result of the paper is stated in Theorem \ref{InstaneousGelation}. However, the proof of Theorem \ref{InstaneousGelation} is presented in Section $3$. The proof of Theorem \ref{InstaneousGelation} depends essentially on Lemmas~\ref{ConLemma4} and~\ref{ConLemma6}.  These lemmas confirm the non-existence of a mass-conserving solutions to \eqref{ConEq1} at any time. The remaining part of Section $4$ is devoted to the mass-conserving solutions to \eqref{ConEq1}. In this section, the solution $\xi$ of \eqref{ConEq1} conserves mass, as shown in Theorem \ref{ConMassTheorem}, for a special form of unbounded coefficients, if it exists.

\section{Preliminaries and main result}
In order to study the possible occurrence of instantaneous gelation of solution to \eqref{ConEq1}--\eqref{ConEqin1}, let us first introduce Banach space $\mathcal{Y}$ with the norm $\|.\|_{\mathcal{Y}}$,
$$\mathcal{Y}=\bigg\{ \xi\in L^{1}(0,\infty):\|\xi\|_{\mathcal{Y}}<\infty\bigg\}\quad \mbox{where}\quad \|\xi\|_{\mathcal{Y}}=\int_{0}^{\infty}(1+\mu)|\xi(\mu,t)|\ d\mu.$$
We also set
$$\mathcal{Y}^{+} = \bigg\{\xi\in \mathcal{Y}\ : \ \xi(\mu)\geq0~\mbox{for}~\mbox{each}~\mu\geq0 \bigg\}.$$

\begin{definition}\label{Condefinition}
Let us assume that $\Lambda$ satisfies hypothesis \textbf{(A)} given below. Then, a function $\xi=\xi(\mu,t)$ is said to be a weak solution to the OHS equation \eqref{ConEq1}--\eqref{ConEqin1} with initial condition $\xi^{\mathrm{in}}$ if
$$0\leq\xi\in\mathcal{C}_w([0,T); L^{1}(\mathbb{R}_{>0}))\cap L^{\infty}(0,T; \mathcal{Y}^{+})$$
such that
\begin{align*}
(\mu,\nu,s)\longmapsto\mu\Lambda(\mu,\nu)\zeta(\mu,s)\zeta(\nu,s)\in L^{1}(\mathbb{R}_{>0}^{2}\times(0,t)),
\end{align*}
 and
\begin{align}\label{Conweaksol}
\int_{0}^{\infty}\omega(\mu)\big[\xi(\mu,t)-\xi(\mu,0)\big]\ d\mu=\int_{0}^{t}\int_{0}^{\infty}
\int_{0}^{\mu}\varpi_1(\mu, \nu) \xi(\mu,s)\xi(\nu,s)\Lambda(\mu,\nu)\ d\nu d\mu ds,
\end{align}
where
\begin{align}\label{Identity1}
 \varpi_1(\mu, \nu) = \nu \varpi'(\mu)-\varpi(\nu),
 \end{align}
for all $\varpi\in \mathcal{W}^{1,\infty}(\mathbb{R}_{>0})$ and first derivative of $\varpi$ is compactly supported. Here,
the spaces $\mathcal{W}^{1,\infty}$ and $\mathcal{C}([c, d]_W; L^{1}(\mathbb{R}_{>0}; d\mu))$ denote the Sobolev space and collection of
continuous functions (in time) with respect to weak topology of $L^{1}\mathbb({R}_{>0}; d\mu)$, respectively.
\end{definition}

Now, we state the following hypothesis on the coagulation kernels $\Lambda$, which is used in the subsequent analysis.
\begin{hypp}
There exists positive constants $\theta_{1},~\theta_{2}$ and $1<\beta<\gamma$ such that
\begin{description}
  \item[\textbf{(A)}]$\theta_{1}(\mu^{\beta}+\nu^{\beta})\leq\Lambda(\mu,\nu)\leq \theta_{2}(1+\mu)^{\gamma} (1+\nu)^{\gamma},\quad (\mu, \nu)\in\mathbb{R}_{>0}^{2}$.
\end{description}
\end{hypp}

Next, we define the moments of the concentration $\xi$ as follows
$$\mathcal{M}^{r}(t):=\int_{0}^{\infty}\mu^{r}\xi(\mu,t)\ d\mu\quad \mbox{and}\quad \mathcal{M}_{m}^{r}(t):=\int_{m}^{\infty}\mu^{r}\xi(\mu,t)\ d\mu,\quad r\in\mathbb{R}, $$
where $\mathcal{M}^{r}(t)$ is called $r^{th}$ moment of the concentration $\xi$.\\

We are now in position to state the main result of this paper.
\begin{thm}[Existence of instantaneous gelation]\label{InstaneousGelation} Assume that coagulation kernel $\Lambda$ satisfies $\textbf{(A)}$. Let $\xi$ be a weak solution to the OHS equations \eqref{ConEq1}--\eqref{ConEqin1} in the sense of Definition \ref{Condefinition}. Then, for all weak solutions of the OHS equations \eqref{ConEq1}--\eqref{ConEqin1}, the gelation occurs instantaneously, i.e., $T_{\mathrm{gel}}=0$.
\end{thm}
To perform the proof of Theorem \ref{InstaneousGelation} , we follow the arguments developed by Carr and da Costa in \cite{Carr:1992} and Banasiak et al.\ in \cite{Banasiak} . According to \cite{Carr:1992} and \cite{Banasiak}, the proof of Theorem~\ref{InstaneousGelation} involves two steps. These steps lead to a contradiction if the gelation time is positive. In Section~$2$, the first step is performed, i.e., if $T_{\mathrm{gel}}>0$, then all moments $\mathcal{M}^{r}(t)$ are finite for all $r\geq1$ and $t\in[0, T_{\mathrm{gel}})$. On the other hand, the second step shows that all moments $\mathcal{M}^{r}(t)$ are finite only on the time interval $(0, t_{r})$, with $t_{r}\rightarrow0$ as $r\rightarrow\infty$.


\section{Existence of instantaneous gelation}

Before proving the occurrence of instantaneous gelation for the OHS equations \eqref{ConEq1}--\eqref{ConEqin1}, we need to show some basic results.
\begin{thm}[Positivity of first moment]\label{theorem3}
 Consider $\xi^{\mathrm{in}}\in\mathcal{Y}^{+}$ such that $\xi^{\mathrm{in}}\neq0$ and $\Lambda>0$ a.e. in $(0, \infty)^{2}$ . Suppose that $\xi$ is a weak solution of \eqref{ConEq1}--\eqref{ConEqin1} on $[0, T)$ in the sense of Definition \ref{Condefinition}. Then, for all $R>0$ and $t>0$,
 $$\int_{R}^{\infty}\mu\xi(\mu,t)\ d\mu>0.$$
\end{thm}
\begin{proof}
Suppose, for the sake of contradiction, that there exists $t_{0}>0$ such that
\begin{align}\label{theorem3eq1}
R_{0}:=\inf\bigg\{R\geq0\ :\ \int_{R}^{\infty}\mu\xi(\mu, t_{0})\ d\mu=0\bigg\}<\infty\quad \mbox{and}\quad \xi(R_{0},s)\neq0,\quad \forall\quad s\in(0,t_{0}).
\end{align}
Now, for all $\mu\in\mathbb{R}_{>0}$ and $R>R_{0}$, we set $\varpi(\mu)=\mu\chi_{(R_{0},R)}(\mu)$ into \eqref{Identity1}. Then, we obtain
\begin{equation*}
 \varpi_{1}(\mu, \nu)=\begin{cases}
0,\ & \text{if}\ (\mu, \nu ) \in (0, R_{0}] \times (0, \mu),\ \\
 \nu, \  &  \text{if}\ (\mu,\nu) \in (R_{0}, R) \times (0, R_{0}],\ \\
 0, \  &  \text{if}\ (\mu,\nu) \in (R_{0}, R) \times (R_{0}, \mu),\ \\
-\nu,\ &  \text{if}\ (\mu, \nu)\in [R, \infty) \times (R_{0}, R),\ \\
0,\ &  \text{if}\ (\mu, \nu)\in [R, \infty) \times [R, \mu).
\end{cases}
\end{equation*}
Putting the values of $\varpi$ and $\varpi_{1}$ into \eqref{Conweaksol}, we have
\begin{eqnarray}\label{InstConEq311}
\lefteqn{\int_{R_{0}}^{R}\mu\xi(\mu,t_{0})\ d\mu }\nonumber\\
&=&\int_{R_{0}}^{R}\mu\xi^{\mathrm{in}}(\mu)\ d\mu+\int_{0}^{t_{0}}\int_{R_{0}}^{R}\int_{0}^{\mu}\nu\Lambda(\mu,\nu)\xi(\mu,s)\xi(\nu,s)\ d\nu d\mu ds\nonumber\\
&&-\int_{0}^{t_{0}}\int_{R_{0}}^{R}\int_{R}^{\infty}\nu\Lambda(\mu,\nu)\xi(\mu,s)\xi(\nu,s)\ d\mu d\nu ds\nonumber\\
&=&\int_{R_{0}}^{R}\mu\xi^{\mathrm{in}}(\mu)\ d\mu+\int_{0}^{t_{0}}\int_{R_{0}}^{R}\mathcal{I}_{1}(\mu, s)\ d\mu ds
+\int_{0}^{t_{0}}\int_{R_{0}}^{R}\mathcal{I}_{2}(\nu, s)\ d\nu ds,
\end{eqnarray}
where
$$\mathcal{I}_{1}(\mu, s)=:\int_{0}^{\mu}\nu\Lambda(\mu,\nu)\xi(\mu,s)\xi(\nu,s)\ d\nu\quad \mbox{and}\quad \mathcal{I}_{2}(\nu, s)=:\int_{R}^{\infty}\nu\Lambda(\mu,\nu)\xi(\mu,s)\xi(\nu,s)\ d\mu.$$
Since $R>R_{0}$ is arbitrary, we deduce from \eqref{theorem3eq1}, \eqref{InstConEq311}, Definition \ref{Condefinition}, the dominated convergence theorem, and the nonnegativity of $\xi^{\mathrm{in}}$ and $(\mathcal{I}_{\wp})_{\wp=1,2}$ that
\begin{equation}\label{InstConEq312}
\begin{array}{rcll}
\xi^{\mathrm{in}}&=&0& a.e.~\mbox{in}\quad (R_{0},\infty),\\
\mathcal{I}_{1}&=&0 & a.e.~\mbox{in}\quad (0, t_{0})\times(R_{0}, \infty), \\
\mathcal{I}_{2}&=&0 & a.e.~\mbox{in}~(0, t_{0})\times(R_{0}, \infty).
\end{array}
\end{equation}
Now, it follows from \eqref{InstConEq312} that
\[
\mu\Lambda(\mu,\nu)\xi(\mu,s)\xi(\nu,s)\chi_{(0,\mu)}(\nu)=0\quad a.e.~\mbox{in}\quad (0, t_{0})\times(R_{0}, \infty)\times(0, \infty).
\]
Consequently, an application of Fubini's theorem yields
\begin{eqnarray}\label{InstConEq31451}
0&=:&\int_{R_{0}}^{\infty}\int_{0}^{\mu}\mu^{2}\xi(\mu,s)\xi(\nu,s)\ d\nu d\mu \nonumber\\
&=&\int_{R_{0}}^{\infty}\int_{0}^{R_{0}}\mu^{2}\xi(\mu,s)\xi(\nu,s)\ d\nu d\mu
+\frac{1}{2}\int_{R_{0}}^{\infty}\int_{R_{0}}^{\mu}\mu^{2}\xi(\mu,s)\xi(\nu,s)\ d\nu d\mu \nonumber\\
&&+\frac{1}{2}\int_{R_{0}}^{\infty}\int_{\mu}^{\infty}\nu^{2}\xi(\mu,s)\xi(\nu,s)\ d\nu d\mu.
\end{eqnarray}
From \eqref{InstConEq31451}, we infer that
\[
\int_{0}^{t_{0}}\bigg(\int_{R_{0}}^{\infty}\nu\xi(\nu,s)d\nu\bigg)^{2}\ ds=0
\]
and thus,
\begin{align}\label{InstConEq315}
\int_{R_{0}}^{\infty}\nu\xi(\nu,s)\ d\nu=0,\quad s\in(0,t_{0}).
\end{align}

Again, we set $\varpi(\mu)=\mu\chi_{(0, R_{0})}(\mu)$  into \eqref{Identity1} for all $\mu\in(0,\infty)$. Then, we obtain
\begin{equation*}
 \varpi_{1}(\mu, \nu)=\begin{cases}
0,\ & \text{if}\ (\mu, \nu ) \in (0, R_{0}) \times (0, \mu),\ \\
-\nu, \  &  \text{if}\ (\mu,\nu) \in (R_{0}, \infty) \times (0, R_{0}),\ \\
0,\ &  \text{if}\ (\mu, \nu)\in (R_{0}, \infty) \times (R_{0}, \mu).
\end{cases}
\end{equation*}
Substituting the values of $\varpi$ and $\varpi_{1}$ into \eqref{Conweaksol} gives
\begin{align}\label{InstConEq3171}
\int_{0}^{R_{0}}\mu\xi(\mu,t_{0})\ d\mu
=&\int_{0}^{R_{0}}\mu\xi^{\mathrm{in}}(\mu)\ d\mu-\int_{0}^{t_{0}}\int_{R_{0}}^{\infty}\int_{0}^{\mu}\nu\Lambda(\mu,\nu)\xi(\mu,s)\xi(\nu,s)\ d\nu d\mu ds.
\end{align}
We infer from \eqref{InstConEq3171} and \eqref{InstConEq31451} that
\begin{align}\label{InstConEq31712}
\int_{0}^{R_{0}}\mu\xi(\mu,t_{0})\ d\mu=\int_{0}^{R_{0}}\mu\xi^{\mathrm{in}}(\mu)\ d\mu.
\end{align}
Now, multiplying the equation \eqref{ConEq1} by $\mu$ and taking integration with respect to $\mu$ from $0$ to $R_{0}$ gives

\begin{align*}
\frac{d}{ds}\int_{0}^{R_{0}}\mu\xi(\mu,s)\ d\mu=-R_{0}\int_{0}^{R_{0}}\nu\xi(R_{0},s)\xi(\nu,s)\Lambda(R_{0}, \nu)\ d\nu.
\end{align*}
Taking integration with respect to time from $0$ to $t_{0}$ yields 
\begin{align*}
\int_{0}^{R_{0}}\mu\xi(\mu,t_{0})\ d\mu=\int_{0}^{R_{0}}\mu\xi^{\mathrm{in}}(\mu)\ d\mu-R_{0}\int_{0}^{t_{0}}\int_{0}^{R_{0}}\nu\xi(R_{0},s)\xi(\nu,s)\Lambda(R_{0},\nu)\ d\nu ds.
\end{align*}
Now, we conclude from definition \eqref{theorem3eq1} of $R_{0}$ and \eqref{InstConEq31712} that
\begin{align*}
 \int_{0}^{t_{0}}\int_{0}^{R_{0}}\nu\xi(\nu,s)\Lambda(R_{0},\nu)\ d\nu ds=0,
 \end{align*}
 from which we readily deduce that
 $$\nu\xi(\nu,s)=0\quad a.e.~\mbox{in}\quad (0, t_{0})\times(0, R_{0})$$
and thus,
 \begin{align*}
 \int_{0}^{t_{0}}\int_{R_{0}/2}^{R_{0}}\nu\xi(\nu,s)\ d\nu ds\leq \int_{0}^{t_{0}}\int_{0}^{R_{0}}\nu\xi(\nu,s)\ d\nu ds=0.
 \end{align*}
 Consequently, it is
 \begin{align}\label{InstConEq3145111}
 \int_{R_{0}/2}^{R_{0}}\nu\xi(\nu,s)\ d\nu=0,\quad s\in(0, t_{0}).
 \end{align}
Now, we obtain from \eqref{InstConEq3145111} and \eqref{InstConEq315} that
\begin{align}\label{InstConEq3145222}
 \int_{R_{0}/2}^{\infty}\nu\xi(\nu,s)\ d\nu=0,\quad s\in(0, t_{0}).
 \end{align}
We conclude, from \eqref{InstConEq3145222} and definition \eqref{theorem3eq1} of $R_{0}$, that $R_{0} = 0$. Then $\xi^{\mathrm{in}}\equiv 0$ according to \eqref{InstConEq312}. This clearly contradicts our assumption.
\end{proof}

\begin{lemma}[Limit behavior for higher moments]\label{lemmapositiv1}
Let $\xi$ be a weak solution to \eqref{ConEq1}--\eqref{ConEqin1} with $\int_{R}^{\infty}\mu\xi(\mu,t)d\mu>0$. Then, we have
$$\lim_{p\rightarrow\infty}\bigg(\int_{0}^{\infty}\mu^{p}\xi(\mu,t)d\mu\bigg)^{1/p}=\infty.$$
\end{lemma}
\begin{proof}
It is for all $l\geq1$ and $t\in(0, T)$
\begin{eqnarray*}
\bigg(\int_{0}^{\infty}\mu^{p}\xi(\mu,t)d\mu\bigg)^{1/p}&\geq& \bigg(\int_{l}^{\infty}\mu^{p}\xi(\mu,t)d\mu\bigg)^{1/p}\\
&\geq& l^{(p-1)/p}\bigg(\int_{l}^{\infty}\mu\xi(\mu,t)d\mu\bigg)^{1/p},
\end{eqnarray*}
so that
$$\lim_{p\rightarrow\infty}\bigg(\int_{0}^{\infty}\mu^{p}\xi(\mu,t)d\mu\bigg)^{1/p}\geq l.$$
The above inequality is valid for all $l\geq1$. Therefore, we conclude that
$$\lim_{p\rightarrow\infty}\left(\mathcal{M}^{p}(t)\right)^{1/p}=\infty.$$
\end{proof}


\begin{lemma}[Integrability of all higher moments]\label{ConLemma4}
Assume that $\Lambda$ satisfies $\textbf{(A)}$ and that for the  initial condition $0\leq\zeta^{\mathrm{in}}\in \mathcal{Y}^{+}$ holds. Assume further that $\xi$ be a solution to \eqref{ConEq1}--\eqref{ConEqin1} on $[0,T)$ s.t. $T_{\mathrm{gel}}\in(0, T]$. Then, for all integers $r\geq1$ and some $t_{0}\in(0, T_{\mathrm{gel}})$, we have
 $$\sup_{t\in [0, t_{0})}(\mathcal{M}^{r}(t))< \infty.$$
  \end{lemma}
  
\begin{proof}
For all $\mu\in(0,\infty)$ and $\lambda\geq1$, set $\varpi(\mu)=\mu\chi_{(0,\lambda)}(\mu)$ into \eqref{Identity1}. Then, we obtain
\begin{equation*}
 \varpi_{1}(\mu, \nu)=\begin{cases}
0,\ & \text{if}\ (\mu, \nu ) \in (0, \lambda) \times (0, \mu),\ \\
- \nu, \  &  \text{if}\ (\mu,\nu) \in [\lambda, \infty) \times (0, \lambda),\ \\
0,\ &  \text{if}\ (\mu, \nu)\in [\lambda, \infty) \times [\lambda, \mu).
\end{cases}
\end{equation*}
Inserting the value of $\varpi$ and $\varpi_{1}$ into \eqref{Conweaksol}, we have for $0 \leq\tau< t< \varepsilon<t_{0}$ that 
\[
\int_{0}^{\lambda}\mu[\xi(\mu,t)-\xi(\mu,\tau)]\ d\mu  =  -\int_{\tau}^{t}\int_{\lambda}^{\infty}\int_{0}^{\lambda}
\nu\Lambda(\mu,\nu)\xi(\mu,s)\xi(\nu,s)\ d\nu d\mu ds.
\]

Using the assumptions on $\Lambda$, we arrive at the inequality
\begin{eqnarray}\label{InstConEq3}
\Gamma_{\lambda}(t)-\Gamma_{\lambda}(\tau)&\leq&-\theta_{1}\int_{\tau}^{t}\left(\int_{\lambda}^{\infty}\int_{0}^{\lambda}
\nu\mu^{\beta}\xi(\mu,s)\xi(\nu,s)\ d\nu d\mu\right) ds\nonumber\\
&\leq&-\theta_{1}\lambda^{\beta-1}\int_{\tau}^{t}\Gamma_{\lambda}(s)\Theta_{\lambda}(s)\ ds,
\end{eqnarray}
 where
 \begin{align*}
\Gamma_{\lambda}(t):=\int_{0}^{\lambda}\mu\xi(\mu,t)\ d\mu \quad \mbox{and}\quad \Theta_{\lambda}(t):=\int_{\lambda}^{\infty}\mu\xi(\mu,t)\ d\mu,
\end{align*}
for $\lambda\geq1$ and $0\leq\tau<t<t_{0}<T_{\mathrm{gel}}$. From $\Gamma_{\lambda}$ and $\Theta_{\lambda}$, we have
$$\Gamma_{\lambda}(t)+\Theta_{\lambda}(t)=\varrho_{0},\quad \forall~t\in[0, T_{\mathrm{gel}}),$$
with $\varrho_0$ defined in \eqref{MassConseDisCon}.
Now, we infer from Dini's monotone convergence theorem \cite{komornik2002functional} that $\mathcal{M}^{1}(t)$ is uniformly convergent in $[0, t_{0}]$. Therefore, there exist a positive integer $p_{0}$ s.t.
$$\frac{\varrho_{0}}{2}\leq \Gamma_{\lambda}(t),\quad \mbox{for}~\mbox{all}\quad \lambda\geq p_{0}\quad \mbox{and}\quad t\in[0,t_{0}].$$

Inserting the above lower bound of $\Gamma_{\lambda}$ into \eqref{InstConEq3} leads to 
\begin{align}\label{FiniteConEq2}
\mathcal{U}_{\lambda}(\tau)=:\Theta_{\lambda}(t)+\frac{1}{2}\varrho_{0}\lambda^{\beta-1}\int_{\tau}^{t}\Theta_{\lambda}(s)\ ds\leq \Theta_{\lambda}(\tau),\quad 0\leq t<\tau\leq t_{0}<T_{\mathrm{gel}}.
\end{align}
Setting $C_{1}=\frac{1}{2}\varrho_{0}\lambda^{\beta-1}$, we infer from \eqref{FiniteConEq2} that
\begin{align*}
\frac{d}{d\tau}\mathcal{U}_{\lambda}(\tau)= C_{1}\Theta_{\lambda}(\tau)\geq C_{1}\mathcal{U}_{\lambda}(\tau).
\end{align*}
Dividing by $\mathcal{U}_{\lambda}(\tau)$, integrating with respect to $\tau$ from $t$ to $t_{0}$ and then applying \eqref{FiniteConEq2}, we obtain
$$\exp(-C_{1}t)\Theta_{\lambda}(t)\leq\exp(-C_{1}t)\mathcal{U}_{\lambda}(t)
\leq\exp(-C_{1}t_{0})\mathcal{U}_{\lambda}(t_{0})\leq\exp(-C_{1}t_{0})\Theta_{\lambda}(t_{0}).$$
Since $\Theta_{\lambda}(t_{0})\leq \varrho_{0}$, we end up with
\begin{align}\label{InstContEq4}
\Theta_{\lambda}(t)\leq\varrho_{0} \exp(-C_{1}(t_{0}-t)),
\end{align}
for all $0\leq t< \tau<t_{0}$, and $\lambda\geq p_{0}$. Next, for fixed $r\geq2$ and $t\in[0, t_{0})$, it is 
 \begin{eqnarray}\label{InstConEq5}
\lefteqn{ \int_{\lambda_{1}}^{\lambda}\mu^{r}\xi(\mu,t)\ d\mu} \nonumber\\ &\leq&\sum_{k=\lambda_{1}}^{\lambda-1}(k+1)^{r-1}\int_{k}^{k+1}\mu\xi(\mu,t)\ d\mu\nonumber\\
 &=&\sum_{k=\lambda_{1}}^{\lambda-1}(k+1)^{r-1}[\Theta_{k}(t)-\Theta_{k+1}(t)]\nonumber\\
 &=&\sum_{k=\lambda_{1}}^{\lambda-1}(k+1)^{r-1}\Theta_{k}(t)-\sum_{k=\lambda_{1}+1}^{\lambda}k^{r-1}\Theta_{k}(t)\nonumber\\
 &\leq& (\lambda_{1}+1)^{r-1}\Theta_{\lambda_{1}}(t)+ \sum_{k=\lambda_{1}+1}^{\lambda-1} \left((k+1)^{r-1}-k^{r-1} \right)
\Theta_{k}(t) \nonumber\\ 
 &\leq&(\lambda_{1}+1)^{r-1}\Theta_{\lambda_{1}}(t)+(r-1)\sum_{k=\lambda_{1}+1}^{\lambda-1}(k+1)^{r-2}\Theta_{k}(t), \quad \lambda> \lambda_{1}\geq p_{0}.
 \end{eqnarray}
Since $\beta>1$ there is $\mathcal{R}\geq p_{0}$ depending on $t, \tau, C_{1}, r$ and $\beta$ such that $r\log(k+1)-C_{1}(\tau-t)\leq0$. Hence, we get from \eqref{InstContEq4} and \eqref{InstConEq5} that
\begin{eqnarray}\label{InstConEq567}
\lefteqn{
\int_{\mathcal{R}}^{\lambda}\mu^{r}\xi(\mu,t)\ d\mu}\nonumber\\
&\leq&(\mathcal{R}+1)^{r-1}\varrho_{0}
+\varrho_{0}(r-1)\sum_{k=\mathcal{R}+1}^{\lambda-1}\exp((r-2)\log(k+1)-\theta_{1}\varrho_{0}k^{\beta-1}(\tau-t)/2)\nonumber\\
&\leq&(\mathcal{R}+1)^{r-1}\varrho_{0}
+\varrho_{0}(r-1)\sum_{k=\mathcal{R}+1}^{\lambda-1}\frac{1}{(k+1)^{2}}.
\end{eqnarray}
The series on the right-hand side of \eqref{InstConEq567} is convergent. Furthermore, we obtain
\begin{align}\label{InstConEq5678}
\int_{0}^{\mathcal{R}}\mu^{r}\xi(\mu,t)d\mu\leq \mathcal{R}^{r-1}\varrho_{0}.
\end{align}
Hence, from \eqref{InstConEq567} and \eqref{InstConEq5678}, the proof of Lemma \ref{ConLemma4} is completed.
\end{proof}


Now, let us focus to establish an equation for $\mathcal{M}^{r}(t)$ for $r\geq2$.
\begin{lemma}[Equation for $\mathcal{M}^{r}(t)$]\label{ConLemma5}
Assume that the coagulation rate $\Lambda$ satisfies $\textbf{(A)}$. Assume further that $\xi$ is a weak solution to \eqref{ConEq1}--\eqref{ConEqin1} on $[0, T)$ with $0\leq\xi^{\mathrm{in}}\in\mathcal{Y}^{+}$ s.t. $T_{\mathrm{gel}}\in(0, T]$. Then, for all $0 \leq \delta\leq t<t_{0}< T_{\mathrm{gel}}$ and $r\geq2$, we have
\begin{align}\label{ConLemma555}
\mathcal{M}^{r}(t)-\mathcal{M}^{r}(\delta)=\int_{\delta}^{t}\int_{0}^{\infty}\int_{0}^{\mu}\omega(\mu,\nu)\xi(\mu,s)\xi(\nu,s)\Lambda(\mu,\nu)\ d\nu d\mu ds,
\end{align}
where $\omega(\mu,\nu)=r\nu\mu^{r-1}-\nu^{r}$.
\end{lemma}

\begin{proof}
 For $\lambda\geq1$, setting $\varpi(\mu)=\mu^{r}\chi_{(0,\lambda)}(\mu)$ into \eqref{Identity1}, we obtain
 \begin{equation*}
 \varpi_{1}(\mu, \nu)=\begin{cases}
r\nu\mu^{r-1}-\nu^{r},\ & \text{if}\ (\mu, \nu ) \in (0, \lambda) \times (0, \mu),\ \\
- \nu^{r}, \  &  \text{if}\ (\mu,\nu) \in [\lambda, \infty) \times (0, \lambda),\ \\
0,\ &  \text{if}\ (\mu, \nu)\in [\lambda, \infty) \times [\lambda, \mu).
\end{cases}
\end{equation*}
 Inserting the value of $\varpi$ and $\varpi_{1}$ into \eqref{Conweaksol}, we get
\begin{eqnarray}\label{ConLemmaeq51}
\int_{0}^{\lambda}\mu^{r}\big[\xi(\mu,t)-\xi(\mu,\delta)\big]\ d\mu
&=&\int_{t}^{\varepsilon}
\int_{0}^{\lambda}\int_{0}^{\mu}(r\nu\mu^{r-1}-\nu^{r})\Lambda(\mu,\nu)\xi(\mu,s)\xi(\nu,s)d\nu d\mu ds\nonumber\\
&&-\int_{t}^{\varepsilon}
\int_{\lambda}^{\infty}\int_{0}^{\lambda}\nu^{r}\Lambda(\mu,\nu)\xi(\mu,s)\xi(\nu,s)d\nu d\mu ds.
\end{eqnarray}
 Let us estimate both terms on the right-hand side of \eqref{ConLemmaeq51}, separately. By assumption $\textbf{(A)}$ and 
 Lemma~\ref{ConLemma4}, we can simplify the first integral term as
\begin{eqnarray}\label{ConLemmaeq52}
\lefteqn{\int_{0}^{\lambda}\int_{0}^{\mu} (r\nu\mu^{r-1}-\nu^{r})\Lambda(\mu,\nu)\xi(\mu,s)\xi(\nu,s)\ d\nu d\mu}\nonumber\\
&\leq& \theta_{2}r\int_{0}^{\lambda}\int_{0}^{\mu}\nu\mu^{r-1}(1+\mu)^{\gamma}(1+\nu)^{\gamma}\xi(\mu,s)\xi(\nu,s)\ d\nu d\mu \nonumber\\
&\leq&2^{2\gamma}\theta_{2}r\int_{0}^{\infty}\int_{0}^{\infty} (\mu^{r-1}+2\mu^{\gamma+r-1}+\mu^{2\gamma+r-1})\nu\xi(\mu,s)\xi(\nu,s)\ d\nu d\mu\nonumber\\
&\leq&2^{2\gamma}\theta_{2}r\varrho_{0}(\|\xi\|_{r-1}+\|\xi\|_{\gamma+r-1}+\|\xi\|_{2\gamma+r-1})\nonumber\\
&\leq&\Omega_{1}<\infty, \quad \mbox{for}\quad 0 \leq \delta\leq t<t_{0}< T_{\mathrm{gel}},
\end{eqnarray}
where $\Omega_{1}$ is a constant independent of $\lambda$ and $s$. Similarly, we conclude from Lemma \ref{ConLemma4} that
\begin{equation}\label{ConLemmaeq53}
\int_{\lambda}^{\infty}\int_{0}^{\lambda}\nu^{r}\Lambda(\mu,\nu)\xi(\mu,s)\xi(\nu,s)\ d\nu d\mu
\leq \Omega_{1}<\infty,\quad \mbox{for}\quad 0 \leq \delta\leq t<t_{0}< T_{\mathrm{gel}}.
\end{equation}

Thanks to \eqref{ConLemmaeq52} and \eqref{ConLemmaeq53}, we may pass to the limit as $\lambda\rightarrow\infty$ in \eqref{ConLemmaeq51} and conclude, from the dominated convergence theorem, that \eqref{ConLemma555} is valid.
\end{proof}


\begin{lemma}[Bound of the time interval before instantaneous gelation occurs]\label{ConLemma6}
Assume that $\textbf{(A)}$ holds and let $\xi$ be a weak solution of \eqref{ConEq1}--\eqref{ConEqin1} on $(0, T]$ with $0\leq\zeta^{\mathrm{in}}\in\mathcal{Y}^{+}$. Suppose $0<\delta< t\leq \tau<t_{0}< T_{\mathrm{gel}}$ where $\delta$ and $\tau$ are fixed and $T_{\mathrm{gel}}\in(0, T]$. Then, for each $r\geq2$, we have
\begin{align}\label{timeineq}
t\leq \delta+\frac{1}{(\beta-1)\theta_{1}\varrho_{0}^{1-\sigma}}\bigg[\frac{1}{\mathcal{M}^{r}(\delta)}\bigg]^{\sigma},
\end{align}
where $\sigma:=(\beta-1)/(r-1)$ and $\beta>1$.
\end{lemma}

\begin{proof}
We observe, from Lemma \ref{ConLemma5} and assumption $\textbf{(A)}$, that
\begin{eqnarray}\label{ConLemma61Time}
 \mathcal{M}^{r}(t)-\mathcal{M}^{r}(\delta)
 &=&\int_{\delta}^{t}\int_{0}^{\infty}\int_{0}^{\mu}(r\nu\mu^{r-1}-\nu^{r})\Lambda(\mu,\nu)\xi(\mu,s)\xi(\nu,s)\ d\nu d\mu ds\nonumber\\
 &\geq&(r-1)\int_{\delta}^{t}\int_{0}^{\infty}\int_{0}^{\mu}\nu\mu^{r-1}\Lambda(\mu,\nu)\xi(\mu,s)\xi(\nu,s)\ d\nu d\mu ds\nonumber\\
 &\geq&(r-1)\theta_{1}\int_{\delta}^{t}\int_{0}^{\infty}\int_{0}^{\infty}\mu\nu^{\beta+r-1}\xi(\mu,s)\xi(\nu,s)\ d\nu d\mu ds,
\end{eqnarray}
where $0<\delta<t<t_{0}<T_{\mathrm{gel}}$. Let us apply H\"{o}lder's inequality to obtain
\begin{eqnarray*}
\mathcal{M}^{r}(t)&=&\int_{0}^{\infty}\mu^{r}\xi(\mu,t)\ d\mu\nonumber\\
&=&\int_{0}^{\infty}\mu^{(r+\beta-1)/(\sigma+1)}\mu^{\sigma/(\sigma+1)}\xi(\mu,t)^{1/(\sigma+1)}
\xi(\mu,t)^{\sigma/(\sigma+1)}\ d\mu\nonumber\\
&\leq&\bigg(\int_{0}^{\infty}\bigg(\mu^{(r+\beta-1)/(\sigma+1)}\xi(\mu,t)^{1/(\sigma+1)}\bigg)^{\sigma+1}\ d\mu\bigg)^{1/(\sigma+1)}\nonumber\\
&& \times\bigg(\int_{0}^{\infty}\bigg(\mu^{\sigma/(\sigma+1)}\xi(\mu,t)^{\sigma/(\sigma+1)}\bigg)^{(\sigma+1)/\sigma}\ d\mu\bigg)^{\sigma/(\sigma+1)}\nonumber\\
&\leq&\bigg(\int_{0}^{\infty}\mu^{r+\beta-1}\xi(\mu,t)\ d\mu\bigg)^{1/(\sigma+1)}
\bigg(\int_{0}^{\infty}\mu\xi(\mu,t)\ d\mu\bigg)^{\sigma/(\sigma+1)}\nonumber\\
&\leq&(\mathcal{M}^{r+\beta-1}(t))^{1/(\sigma+1)}(\varrho_0)^{\sigma/(\sigma+1)},
\end{eqnarray*}
where $\sigma=(\beta-1)/(r-1)$. We notice, from the above inequality, that
\begin{align}\label{ConHolderineq}
(\mathcal{M}^{r}(t))^{\sigma+1}(\varrho_0)^{-\sigma}\leq&\mathcal{M}^{r+\beta-1}(t).
\end{align}

Consequently, from \eqref{ConLemma61Time} and \eqref{ConHolderineq}, we obtain
\begin{align}\label{ConLemmaeq61}
\mathcal{M}^{r}(t)
\geq \mathcal{M}^{r}(\delta)+(r-1)\theta_{1}\varrho_{0}^{1-\sigma}\int_{\delta}^{t}(\mathcal{M}^{r}(s))^{\sigma+1}ds,\quad 0<\delta<t<t_{0}<T_{\mathrm{gel}}.
\end{align}
Now, we introduce
\begin{align}\label{ConLemmaeq6661}
\mathcal{Q}^{r}(t):=\mathcal{M}^{r}(\delta)+(r-1)\theta_{1}\varrho_{0}^{1-\sigma}\int_{\delta}^{t}(\mathcal{M}^{r}(s))^{\sigma+1}\ ds.
\end{align}
From \eqref{ConLemmaeq61} and \eqref{ConLemmaeq6661}, we get
\begin{align}\label{ConLemmaeq666}
\mathcal{M}^{r}(t)\geq \mathcal{Q}^{r}(t).
\end{align}
Next, we infer from \eqref{ConLemmaeq61}, and \eqref{ConLemmaeq666} that 
\[
\frac{d}{dt}(\mathcal{Q}^{r}(t)) = \theta_{1}\varrho_{0}^{1-\sigma}(r-1)(\mathcal{M}^{r}(t))^{\sigma+1} \ge \theta_{1}\varrho_{0}^{1-\sigma}(r-1)(\mathcal{Q}^{r}(t))^{\sigma+1},\quad t\in[\delta,t_{0}).
\] 
Now, taking integration with respect to $t$ from $\delta$ to $t_{0}$, we have
\[
\int_{\delta}^{t_{0}}\frac{1}{(\mathcal{Q}^{r}(t))^{\sigma+1}}\ d\mathcal{Q}^{r}(t) \ge \theta_{1}\varrho_0^{1-\sigma}(r-1)\int_{\delta}^{t_{0}}\ dt, 
\]
so that 
\begin{equation} \label{Converlimit}
0\leq(\mathcal{Q}^{r}(t_{0}))^{-\sigma}\le(\mathcal{Q}^{r}(\delta))^{-\sigma}+\sigma\theta_{1}\varrho_{0}^{1-\sigma}(r-1)(\delta-t_{0}).
\end{equation}
The definition of $\sigma$ implies that 
 \begin{align*}
(Q^{r}(t))^{-\sigma}\ge\frac{1}{(\mathcal{Q}^{r}(\delta))^{\sigma}}+(\beta-1)\theta_{1}\varrho_{0}^{1-\sigma}(\delta-t)\geq0.
\end{align*}
Consequently, we obtain
\begin{align*}
t\leq  \delta+\frac{1}{(\beta-1)\theta_{1}\varrho_{0}^{(r-\beta)/(r-1)}}\bigg[\frac{1}{\mathcal{Q}^{r}(\delta)}\bigg]^{\sigma},
\end{align*}
from which \eqref{timeineq} is easily deduced.
\end{proof}

\textbf{Proof of the existence of instantaneous gelation:}~Assume for contradiction that $T_{\mathrm{gel}}\in(0,\infty]$. From Lemma \ref{ConLemma6}, we get an estimate for a solution of \eqref{ConEq1}--\eqref{ConEqin1} on $[0, T)$ as
\begin{align}\label{LemmaCon62}
t\leq\delta+\frac{1}{(\beta-1)\theta_{1}\varrho_{0}^{(r-\beta)/(r-1)}}\bigg[\frac{1}{\mathcal{M}^{r}(\delta)}\bigg]_{,}^{(\beta-1)/(r-1)}
\end{align}
where $r\geq2$ and $0\leq\delta<t<T_{\mathrm{gel}}$. Now, we may pass to the limit as $r\rightarrow\infty$ in \eqref{LemmaCon62} and conclude from Lemma~\ref{lemmapositiv1} that $t\leq\delta$ for $\delta\in(0,t)$. Hence a contradiction is constructed. Therefore, it is  $T_{\mathrm{gel}}=0$.\\


In the following section, we show the mass conservation property of the solution to the OHS model if it exists.

\section{Mass-conserving solutions}
To prove the equality \eqref{MassConseDisCon}, we consider the following form of coagulation kernels:
\begin{equation}\label{Conhyppmcs1}
   \Lambda(\mu,\nu)=\varphi(\mu)+\varphi(\nu)+\Psi(\mu,\nu),\quad (\mu,\nu)\in(0,\infty)^{2},
\end{equation}
where $\varphi(\mu)=\theta_{1}\mu^{\beta}$ and $0\leq\Psi(\mu,\nu)\leq\mathcal{K}(\mu+\nu)$ for all $\beta>1$ and some constants $\theta_{1},~\mathcal{K}>0$. Here, we notice that the class of coagulation kernel \eqref{Conhyppmcs1} satisfies hypothesis $\textbf{(A)}$.

\begin{thm}[Mass conservation]\label{ConMassTheorem}
Assume that the coagulation kernel $\Lambda$ satisfies \eqref{Conhyppmcs1}. If $\xi$ is a weak solution to \eqref{ConEq1}--\eqref{ConEqin1} with $0\leq \xi^{\mathrm{in}}\in\mathcal{Y}^{+}$ on $[0, T)$, where $T\in(0, \infty]$, then $\xi$ satisfies mass conservation property \eqref{MassConseDisCon} for all $0\leq t <T$.
\end{thm}

\begin{proof}
First, let us assume that $\xi$ is a weak solution to the coagulation equation \eqref{ConEq1} on $[0, T)$ in the sense of
Definition~\ref{Condefinition}. Since $\xi^{\mathrm{in}}\not\equiv0$ and $\mathcal{M}^{0}(\zeta)\in \mathcal{C}([0, T))$, there are $t_{0}\in[0, T)$ and $\varepsilon_{1}>0$ such that
\begin{align}\label{lemmacon41}
\mathcal{M}^{0}(\xi)\geq\varepsilon_{1},\quad t\in[0,t_{0}].
\end{align}
From \eqref{Conhyppmcs1}, \eqref{lemmacon41} and Definition \ref{Condefinition}, we get
\begin{align*}
\varepsilon_{1}\int_{0}^{t_{0}}\mathcal{M}^{\iota}(\xi(t))\ dt
\leq&\int_{0}^{t_{0}}\mathcal{M}^{0}(\xi(t))\mathcal{M}^{\iota}(\xi(t))\ dt\nonumber\\
\leq&\frac{1}{2}\int_{0}^{t_{0}}\int_{0}^{\infty}\int_{0}^{\infty}\Lambda(\mu,\nu)\xi(t)\xi(t)\ dt<\infty,
\end{align*}
and this implies
\begin{align}\label{boundCon1}
\mathcal{M}^{\iota}(\xi)\in L^{1}(0, t_{0})\quad \mbox{for}\quad \mbox{all}~\iota\geq0.
\end{align}
For all $\mu\in(0,\infty)$, set $\varpi(\mu)=\mu\chi_{(0,\lambda)}(\mu)$ into \eqref{Identity1} to get
 \begin{equation*}
 \varpi_{1}(\mu, \nu)=\begin{cases}
0,\ & \text{if}\ (\mu, \nu ) \in (0, \lambda) \times (0, \mu),\ \\
- \nu, \  &  \text{if}\ (\mu,\nu) \in [\lambda, \infty) \times (0, \lambda),\ \\
0,\ &  \text{if}\ (\mu, \nu)\in [\lambda, \infty) \times [\lambda, \mu).
\end{cases}
\end{equation*}
Inserting the value of $\varpi$ and $\varpi_{1}$ into \eqref{Conweaksol}, we end up with
\begin{align}\label{ConMassconser1}
\int_{0}^{\lambda}\mu\big[\xi(\mu,t)-\xi(\mu,0)]\\big d\mu=-&\int_{0}^{t}\int_{\lambda}^{\infty}\int_{0}^{\lambda}\nu\Lambda(\mu,\nu) \zeta(\mu,s)\xi(\nu,s)\ d\nu d\mu ds.
\end{align}
We simplify the right-hand side of \eqref{ConMassconser1}, using the properties of the coagulation kernel, 
 \begin{eqnarray}\label{Conmassconser4}
\lefteqn{\int_{0}^{t}\int_{\lambda}^{\infty}\int_{0}^{\lambda}\nu\Lambda(\mu,\nu) \xi(\mu,s)\xi(\nu,s)\ d\nu d\mu ds
}\nonumber \\
&\leq&2(\theta_{1}+\mathcal{K})\int_{0}^{t}\int_{\lambda}^{\infty}\int_{0}^{\lambda}\nu\mu^{\beta}\xi(\mu,s)\xi(\nu,s)\ d\nu d\mu ds\nonumber\\
&\leq&2(\theta_{1}+\mathcal{K})\sup_{t\in[0,T)}\|\xi\|_{\mathcal{Y}}\int_{0}^{t}\int_{\lambda}^{\infty}\mu^{\beta}\xi(\mu,s) \ d\mu ds.
\end{eqnarray}
Finally, \eqref{boundCon1}, \eqref{Conmassconser4}, and the dominated convergence theorem allow to conclude that the solution $\zeta$ satisfies \eqref{MassConseDisCon} as $\lambda\rightarrow \infty$ in \eqref{ConMassconser1}.
\end{proof}

\begin{corollary}[Nonexistence of a weak solution.]\label{Concor1}
Suppose the coagulation kernel $\Lambda$ satisfies \eqref{Conhyppmcs1}. Let $\xi^{\mathrm{in}}\in\mathcal{Y}^{+}$ be  non-trivial initial data. Then, the equation \eqref{ConEq1}--\eqref{ConEqin1} has no weak solution, defined in $[0, T)$, for any $T>0$.
\end{corollary}

\begin{proof}
From Theorem \ref{ConMassTheorem}, we get that the solution $\xi(t)$ satisfies the mass conserving property \eqref{MassConseDisCon} for the class of coagulation kernels \eqref{Conhyppmcs1}. The class of coagulation kernels \eqref{Conhyppmcs1} satisfies the condition $\textbf{(A)}$. Hence, Theorem \ref{InstaneousGelation} implies that $T_{\mathrm{gel}}=0$. These both statements are completely opposite to each other. Hence the equation \eqref{ConEq1}--\eqref{ConEqin1} with \eqref{Conhyppmcs1} has no solution. This completes the proof.
\end{proof}

\section*{Acknowledgments}
This work was partially supported by Department of Science \& Techonolgy (DST), India- Deutscher Akademischer Austauschdienst (DAAD) within the Indo-German joint project entitled ``Analysis and Numerical Methods for Population Balance Equations''.

\end{document}